\documentclass[
final
]{dmtcs-episciences}

\usepackage[utf8]{inputenc}
\usepackage{subfigure}

\usepackage{rorabaugh}

\usepackage[style=authoryear,backend=bibtex]{biblatex}
\renewbibmacro{in:}{}
\author{Joshua Cooper\affiliationmark{1}
  \and Danny Rorabaugh\affiliationmark{2}}
\title[Asymptotic Density of Zimin Words]
     {Asymptotic Density of Zimin Words}
\affiliation{
  Department of Mathematics, University of South Carolina, Columbia\\
  Department of Mathematics and Statistics, Queen's University, Kingston}
\keywords{free words, homomorphism density, limit theory, ZImin}
\received{2015-10-15}
\revised{2016-02-29}
\accepted{2016-03-03}
\begin{document}
\publicationdetails{18}{2016}{3}{3}{1302}
\maketitle
\begin{abstract}
Word $W$ is an \emph{instance} of word $V$ provided there is a homomorphism $\phi$ mapping letters to nonempty words so that $\phi(V) = W$. For example, taking $\phi$ such that $\phi(c)=fr$, $\phi(o)=e$ and $\phi(l)=zer$, we see that ``freezer'' is an instance of ``cool''.

Let $\II_n(V,[q])$ be the probability that a random length $n$ word on the alphabet $[q] = \{1,2,\cdots q\}$ is an instance of $V$. Having previously shown that $\lim_{n \rightarrow \infty} \II_n(V,[q])$ exists, we now calculate this limit for two Zimin words, $Z_2 = aba$ and $Z_3 = abacaba$.
\end{abstract}

\section{Introduction}


Our present interest is in words--not the linguistic units with lexical value, but rather strings of symbols or letters.
We are interested in words as abstract discrete structures.
In particular, we are investigating elements of a free monoid. 
A monoid is an algebraic structure consisting of a set, an associative binary operation on the set, and an identity element. 
A free monoid is defined over some generating set of elements, which we view as an alphabet of letters. 
Its binary operation is simply concatenation, its elements--called free words--are all finite strings of letters, and its identity element is the empty word (generally denoted with $\varepsilon$ or $\lambda$). 
Often, the operation of a monoid is called multiplication, so it is fitting that a ``subword'' of a free word is called a ``factor.'' 
For example, in the free monoid over alphabet $\{a,b,c,d,r\}$, the word $cadabra$ is a factor of $abracadabra$ because $abracadabra$ is the product of $abra$ and $cadabra$. 

\subsection{Combinatorial Limit Theory}

In an era of massive technological and computational advances, we have large systems for transportation, communication, education, and commerce (to name a few examples). 
We also possess massive quantities of information in every part of life. 
Therefore, in many applications of discrete mathematics, the useful theory is that which is relevant to arbitrarily large discrete structures. 
For example, graphs can be used to model a computer network, with each vertex representing a device and each edge a data connection between devices. 
The most well-known computer network, the Internet, consists of billions of devices with constantly changing connections; one cannot simply create a database of all billion-vertex graphs and their properties. 

We use the term ``combinatorial limit theory'' in general reference to combinatorial methods which help answer the following question: What happens to discrete structures as they grow large? 
In the combinatorial limit theory of graphs, major recent developments include the flag algebras of \textcite{R-07} and the graph limits of Borgs, Chayes, Freedman, Lov\'asz, Schrijver, S\'os, Szegedy, Vesztergombi, etc. (see \cite{L-12}).
Given the fundamental reliance of these methods on graph homomorphisms and graph densities, we strive to apply the same ideas to free words, or henceforth, simply ``words.''

\subsection{Definitions} \label{words}

\begin{defn} \label{defn:word}
	For a fixed set $\Sigma$, called an \emph{alphabet},  denote with $\Sigma^*$ the set of all finite words formed by concatenation of elements of $\Sigma$, called \emph{letters}.
	Words in $\Sigma^*$ are called \emph{$\Sigma$-words}.
	The set of length-$n$ $\Sigma$-words is denoted with $\Sigma^n$.
	The \emph{empty word}, denoted $\varepsilon$, consisting of zero letters, is a $\Sigma$-word for any alphabet $\Sigma$.
\end{defn}

The set $\Sigma^*$, together with the associative binary operation of concatenation and the identity element $\varepsilon$, forms a free monoid.
We denote concatenation with juxtaposition.
Generally we use natural numbers or minuscule Roman letters as letters and majuscule Roman letters (especially $T,U,V,W,X,Y,$ and $Z$) to name words.
Majuscule Greek letters (especially $\Gamma$ and $\Sigma$) name alphabets, though for a standard $q$-letter alphabet, we frequently use the set $[q] = \{1, 2, \ldots, q \}$.

\begin{ex}
	Alphabet $[3]$ consists of letters 1, 2, and 3. 
	The set of $[3]$-words is 
	\[\{1,2,3\}^* = \{\varepsilon, 1, 2, 3, 11, 12, 13, 21, 22, 23, 31, 32, 33, 111, 112, 113, 121, \ldots \}.\]
\end{ex}

\begin{defn} \label{defn:letter}
	A word $W$ is formed from the concatenation of finitely many letters.
	If letter $x$ is one of the letters concatenated to form $W$, we say $x$ \emph{occurs in} $W$, or $x \in W$.
	For natural number $n \in \N$, an $n$-fold concatenation of word $W$ is denoted $W^n$.
	The \emph{length} of word $W$, denoted $|W|$, is the number of letters in $W$, counting multiplicity.
	$\L(W)$, the \emph{alphabet generated by} $W$, is the set of all letters that occur in $W$.
	For $q \in \N$, word $W$ is \emph{$q$-ary} provided $|L(W)| \leq q$.
	We use $||W||$ to denote the number of letter recurrences in $W$, so $||W|| = |W| - |L(W)|$.
\end{defn}

\begin{ex}
	Let $W = bananas$.
	Then $a,b \in W$, but $c \not\in W$. 
	Also $|W| = 7$, $L(W) = \{a,b,n,s\}$, and $||W|| = 3$. 
	
	For the empty word, we have $|\varepsilon| = 0$, $L(\varepsilon) = \emptyset$, and $||\varepsilon|| = 0$.
\end{ex} 

\begin{defn} \label{defn:factor}
	Word $W$ has $\binom{|W|+1}{2}$ (nonempty) \emph{substrings}, each defined by an integer pair $(i,j)$ with $0\leq i < j \leq |W|$.
	Denote with $W[i,j]$ the word in the $(i,j)$-substring, consisting of $j-i$ consecutive letters of $W$, beginning with the $(i+1)$-th.

	Word $V$ is a \emph{factor} of $W$, denoted $V \leq W$, provided $V = W[i,j]$ for some integers $i$ and $j$ with $0\leq i < j \leq |W|$; equivalently, $W = SVT$ for some (possibly empty) words $S$ and $T$.
\end{defn}

\begin{ex}
	$nana \leq nana \leq bananas$, with $nana = nana[0,4] = bananas[2,6]$.
\end{ex}


\begin{defn}
	For alphabets $\Gamma$ and $\Sigma$, every (monoid) homomorphism $\phi: \Gamma^* \rightarrow \Sigma^*$ is uniquely defined by a function $\phi:\Gamma \rightarrow \Sigma^*$. 
	We call a homomorphism \emph{nonerasing} provided it is defined by $\phi:\Gamma \rightarrow \Sigma^* \setminus \{\varepsilon\}$; that is, no letter maps to $\varepsilon$.
\end{defn}

\begin{ex}
	Consider the homomorphism $\phi: \{b,n,s,u\}^* \rightarrow \{m,n,o,p,r,v\}^*$ defined by Table~\ref{table:exFunction}. 
	Then $\phi(sun) = moon$ and $\phi(bus) = vroom$.
\end{ex}

	\begin{table}[ht]
	\centering
		\caption{Example of a nonerasing function.} \label{table:exFunction}
		
		\begin{tabular}{c | c | c | c | c} 
			$x$ & $b$ & $n$ & $s$ & $u$\\ \hline
			$\phi(x)$ & $vr$ & $n$ & $m$ & $oo$
		 \end{tabular}
	\end{table}

\newpage

\begin{defn} \label{defn:instance}
	$U$ is an \emph{instance of $V$}, or a \emph{$V$-instance}, provided $U = \phi(V)$ for some nonerasing homomorphism $\phi$; equivalently,
	\begin{itemize}
		\item $V = x_0x_1 \cdots x_{m-1}$ where each $x_i$ is a letter;
		\item $U = A_0A_1 \cdots A_{m-1}$ with each word  $A_i \neq \varepsilon$ and $A_i = A_j$ whenever $x_i=x_j$.
	\end{itemize}
	$W$ \emph{encounters} $V$, denoted $V \preceq W$, provided $U \leq W$ for some $V$-instance $U$.
	If $W$ fails to encounter $V$, we say $W$ \emph{avoids} $V$.
\end{defn}

To help distinguish the encountered word and the encountering word, ``pattern'' is elsewhere used to refer to $V$ in the encounter relation $V \preceq W$. 
Also, an instance of a word is sometimes called a ``substitution instance'' and ``witness'' is sometimes used in place of encounter.

\begin{defn} \label{def:unavoidable}
	A word $V$ is \emph{unavoidable} provided, for any finite alphabet, there are only finitely many words that avoid $V$.
\end{defn}

The first classification of unavoidable words was by \textcite{BEM-79}. Three years later, Zimin published a fundamentally different classification of unavoidable words (\cite{Z-82} in Russian, \cite{Z-84} in English).

\begin{defn} \label{defn:Zimin}
	Define the \emph{$n$-th Zimin word} recursively by $Z_0 := \varepsilon$ and, for $n \in \N$, $Z_{n+1} = Z_nx_nZ_n$. Using the English alphabet rather than indexed letters:
	\[Z_1 = \textbf{a}, \quad Z_2 = a\textbf{b}a, \quad Z_3 = aba\textbf{c}aba, \quad Z_4 = abacaba\textbf{d}abacaba, \quad \ldots . \]
\end{defn}
Equivalently, $Z_n$ can be defined over the natural numbers as the word of length $2^n-1$ such that the $i$-th letter, $1 \leq i < 2^n$, is the 2-adic order of $i$.
	
\begin{thm}[\cite{Z-84}] \label{thm:Zimin}
	A word $V$ with $n$ distinct letters is unavoidable if and only if $Z_n$ encounters $V$.
\end{thm}

With Zimin's concise characterization of unavoidable words, a natural combinatorial question follows: How long must a $q$-ary word be to guarantee that it encounters a given unavoidable word? 
Define $\f(n,q)$ to be the smallest integer $M$ such that every $q$-ary word of length $M$ encounters $Z_n$.

In 2014, three preprints by different authors appeared, each independently proving bounds for $\f(n,q)$: \textcite{CR-14}, \textcite{T-14}, and \textcite{RS-15}.

\section{Asymptotic Probability of Being Zimin} \label{ASYMP}

\begin{defn} \label{defn:density}
	Let $\mathbb{I}_n(V,q)$ be the probability that a uniformly randomly selected length-$n$ $q$-ary word is an instance of $V$. 
That is, 
\[
\II_n(V,q) =\frac{|\{W \in [q]^n \mid \phi(V)=W \mbox{ for some nonerasing homomorphism } \phi:{\rm L}(V)^* \rightarrow [q]^*\}|}{ q^{n}}.
\]

Denote $\II(V,q) = \lim_{n\rightarrow \infty}\mathbb{I}_n(V,q)$.
\end{defn}

\textcite{CR-15} prove that $\II(V,q)$ exists for any word $V$.
Moreover, they establish the following dichotomy for $q \geq 2$: $\II(V,q) = 0$ if and only if $V$ is doubled (that is, every letter in $V$ occurs at least twice).
Trivially, if $V$ is composed of $k$ distinct, nonrecurring letters, then $\II_n(V,[q])=1$ for $n\geq k$, so $\II(V,q) = 1$.
But if $V$ contains at least one recurring letter, it becomes a nontrivial task to compute $\II(V,q)$.
We have from previous work the following bounds for the instance probability of Zimin words.

\begin{cor}
	For $n,q \in \Z^+$, 
	\[ q^{-2^n + n + 1} \leq \II(Z_n,q) \leq \prod_{j = 1}^{n-1} \frac{1}{q^{(2^j-1)} - 1}. \]
\end{cor}

\begin{proof}
For the lower bound, note that $||Z_n|| = |Z_n| - |\L(Z_n)| = (2^n-1) - (n)$. 
Theorem~3.3 from \textcite{CR-15} tells us that for all $q \in \Z^+$ and nondoubled $V$, $\II(V,q) \geq q^{-||V||}$.

For the upper bound, observe that the $n$ letters occurring in $Z_n$ have the following multiplicities: $\ang{r_j = 2^j : 0 \leq j < n}$. 
Since there is exactly one nonrecurring letter in $Z_n$, $r_0 = 2^0 = 1$, Theorem~4.14 from \textcite{R-15} provides an upper bound of $\prod_{j = 1}^{n-1} \frac{1}{q^{(r_j-1)} - 1}$.
\end{proof}

A nice property of these bounds is that they are asymptotically equivalent as $q \rightarrow \infty$.
For some specific $V$, we can do better.
Presently, we provide infinite series for computing the asymptotic instance probability $\II(V,q)$ for two Zimin words, $V = Z_2 = aba$ (Section~\ref{IZ2}) and $V = Z_3 = abacaba$ (Section~\ref{IZ3}).
Table~\ref{table:Z2Z3} below gives numerical approximations for $2 \leq q \leq 6$. 
Our method also provides upper bounds on  $\II(Z_n,q)$ for general $n$ (Section~\ref{IZn}).

\begin{table}[ht]
\centering

    	\caption{Approximate values of $\II(Z_2,q)$ and $\II(Z_3,q)$ for $2 \leq q \leq 6$.} \label{table:Z2Z3}

    	\begin{tabular}{c}
	$\begin{array}{c | c | c | c | c | c | c }
    		q & 2 & 3 & 4 & 5 & 6 & \cdots \\ \hline
        		\II(Z_2,q) & 0.7322132 & 0.4430202 & 0.3122520 & 0.2399355 & 0.1944229 & \cdots\\ \hline
    		\II(Z_3,q) & 0.1194437 & 0.0183514 & 0.0051925 & 0.0019974 & 0.0009253 & \cdots\\
	\end{array}$
    	\end{tabular}

\end{table}

\section{Calculating \texorpdfstring{$\II(Z_2,q)$}{the asymptotic instance probability of Z2}} \label{IZ2}


\begin{defn}
	Nonempty word $V$ is a \emph{bifix} of word $W$ provided $W = VA = BV$ for some nonempty words $A$ and $B$; that is, $V$ is both a proper prefix and suffix of $W$.	
	Moreover, if bifix $V$ is an instance of word $Z$, then $V$ is a \emph{$Z$-bifix} of $W$.
	If word $W$ has no bifixes, $W$ is \emph{bifix-free}. 
	If $W$ has no $Z$-bifix, $W$ is \emph{$Z$-bifix-free}. 
\end{defn}

\begin{lem}
	If word $W$ has a bifix, then it has a bifix of length at most $\lfloor |W|/2 \rfloor$.
\end{lem}

\begin{proof}
	Let $W$ be a word with minimal-length bifix of length $k$, $\lfloor |W|/2 \rfloor < k < |W|$. Then we can write $W = W_1W_2W_3$ where $W_1W_2 = W_2W_3$ and $|W_1W_2| = k = |W_2W_3|$. But then $W$ has bifix $W_2$ with $|W_2| < k$, which contradicts our selection of the shortest bifix of $W$.
\end{proof}

Although some words are neither $Z_2$-instances nor bifix-free, the proportion of such words is asymptotically $0$. 
Hence, $1-\II(Z_2,q)$ was previously computed by \textcite{N-73} as the asymptotic probability that a word is bifix-free. 
Equivalently, in a paper of \textcite{GO-81} on the period, or overlap, of words, $1-\II(Z_2,q)$ was computed as the proportion of strings with no period. 
Rather than restate these results, we reformulate them presently for completeness and as a warm-up for calculating $\II(Z_3,q)$. 

Let $a_\ell = a_\ell^{(q)}$ be the number of bifix-free $q$-ary strings of length $\ell$. For $q=2$, this is sequence oeis.org/A003000; for $q=3$, oeis.org/A019308 \parencite{OEIS}.

\begin{lem}[\cite{N-73}, Theorem 1]
	$a_\ell = a_\ell^{(q)}$ has the following recursive definition:
	\begin{eqnarray*}
		a_0 & = & 0;\\
		a_1 & = & q;\\
		a_{2k} & = & qa_{2k-1} - a_k;\\
		a_{2k+1} & = & qa_{2k}.
	\end{eqnarray*}
\end{lem}

\begin{proof}
	Fix a $q$-letter alphabet. 
	Let $W = UV$ be a bifix-free word with $|U| = \ceil{\frac{|W|}{2}}$ and $|V| = \floor{\frac{|W|}{2}}$. 
	Suppose $UaV$ has a bifix for some letter $a$. 
	Then by the lemma, $UaV$ has a bifix of length at most $|UaV|/2$. 
	But $W$ is bifix free, so the only possibility is $U = aV$.

	Therefore, for every bifix-free word of length $2k$ there are $q$ bifix-free words of length $2k+1$. 
	For every bifix-free word of length $2k-1$, there are $q$ bifix-free words of length $2k$, with exception of the the length-$2k$ words that are the square of a bifix-free word of length $k$.
\end{proof}

\begin{thm} \label{thm:IZ2}
	For $q \geq 2$,
	\begin{eqnarray*}
		\II(Z_2,q) &=& \sum_{j=0}^{\infty} \frac{(-1)^jq^{\left(1-2^{j+1}\right)}}{\prod_{k=0}^{j} \left(1 - q^{\left(1-2^{k+1}\right)}\right)}.
	\end{eqnarray*}
\end{thm}

\begin{proof}
	Since $a_\ell = a_\ell^{(q)}$ counts bifix-free words, the number of $q$-ary words of length $M$ that are $Z_2$-instances is (without double-count)
	\[\sum_{\ell=0}^{\lceil M/2 \rceil -1} a_\ell q^{M - 2\ell},\]
	so the proportion of $q$-ary words of length $M$ that are $Z_2$-instances is
	\[\frac{1}{q^M} \sum_{\ell=0}^{\lceil M/2 \rceil -1} a_\ell q^{M - 2\ell} = \sum_{\ell=0}^{\lceil M/2 \rceil -1} \frac{a_\ell}{q^{2\ell}}.\]
	Therefore $\II(Z_2,q)  = f(1/q^2)$, where $f(x) = f^{(q)}(x)$ is the generating function for $\{a_\ell\}_{\ell = 0}^\infty$:	\[f(x) = \sum_{\ell=0}^{\infty} a_\ell x^\ell.\]
	From the recursive definition of $a_\ell$, we obtain the functional equation
	\begin{eqnarray} \label{Z2func}
		f(x) = qx + qxf(x) - f(x^2).
	\end{eqnarray}
	Solving \eqref{Z2func} for $f(x)$ gives \[f(x) = \frac{qx - f(x^2)}{1-qx} = \cdots = \sum_{j=0}^{\infty} \frac{(-1)^jqx^{2^j}}{\prod_{k=0}^{j} (1 - qx^{2^k})}.\]
\end{proof}

\begin{cor} \label{cor:IZ2}
	For $q \geq 2$:
	\[\frac{1}{q} < \II(Z_2,q) < \frac{1}{q-1}.\]
	Moreover, as $q \rightarrow \infty$,
	\[\II(Z_2,q) = \frac{1}{q-1} - \frac{1+o(1)}{q^3}.\]
\end{cor}

\begin{proof}
	The lower bound follows from the fact that a word of length $M>2$ is a $Z_2$-instance when the first and last character are the same. 
	This occurrence has probability $1/q$. 
	Note that $f^{(q)}(q^{-2})$ is an alternating series.
	Moreover, the terms in absolute value are monotonically approaching 0; the routine proof of monotonicity can be found in the appendices  (Lemma~\ref{lemF}).
	Hence, the partial sums provide successively better upper and lower bounds: 
	\begin{eqnarray*}
		 f^{(q)}\left(\frac{1}{q^2}\right) & = & \sum_{j=0}^{\infty} \frac{(-1)^j\left(q^{1-2^{j+1}}\right)}{\prod_{k=0}^{j} \left(1 - \left(q^{1-2^{k+1}}\right)\right)};\\
		\\
		f^{(q)}\left(\frac{1}{q^2}\right) &>& \sum_{j=0}^{1} \frac{(-1)^j\left(q^{1-2^{j+1}}\right)}{\prod_{k=0}^{j} \left(1 - \left(q^{1-2^{k+1}}\right)\right)}\\
		& = &  \frac{1/q}{1-1/q} - \frac{1/q^3}{(1 - 1/q)(1 - 1/q^3)}\\
		& = & \frac{1}{q-1} - \frac{1+o(1)}{q^3};\\
		\\
		f^{(q)}\left(\frac{1}{q^2}\right) & < & \sum_{j=0}^{2} \frac{(-1)^jq\left(\frac{1}{q^2}\right)^{2^j}}{\prod_{k=0}^{j} \left(1 - q\left(\frac{1}{q^2}\right)^{2^k}\right)}\\
		& = &\frac{1}{q-1} -\frac{1+o(1)}{q^3} + \frac{1/q^5}{(1 - 1/q)(1 - 1/q^3)(1 - 1/q^5)}\\
		& = &\frac{1}{q-1} -\frac{1+o(1)}{q^3} +\frac{O(1)}{q^5}.
	\end{eqnarray*}
\end{proof}

\begin{table}[ht]
\centering

	\caption{Approximate values of $\II(Z_2,q)$ for $2 \leq q \leq 8$.}

	\def\arraystretch{1.3}
	\begin{tabular}{c | c c c c c c c c c}

			$q$ & 2 & 3 & 4 & 5 & 6 & 7 & 8 \\ \hline 
			$q^{-1}$ & 0.50000 & .33333 & .25000 & .20000 & .16667 & .14286 & .12500\\ \hline
			$\II(Z_2,q)$ & 0.73221 & .44302 & .31225 & .23994 & .19442 & .16326 & .14062\\ \hline
			$(q-1)^{-1} - q^{-3}$ & 0.87500 & .46296 & .31771 & .24200 & .19537 &.16375 &.14090\\ \hline
			$(q-1)^{-1}$ & 1.00000 & .50000 & .33333 & .25000 & .20000 & .16667 & .14286\\
	\end{tabular}

\end{table}


%
%
%

\section{Calculating \texorpdfstring{$\II(Z_3,q)$}{the asymptotic instance probability of Z3}} \label{IZ3}

Will use similar methods to compute $\II(Z_3,q)$. 
To avoid unnecessary subscripts and superscripts, assume throughout this section that we are using a fixed alphabet with $q>1$ letters, unless explicitly stated otherwise.
Since $Z_2$ has more interesting structure than $Z_1$, there are more cases to consider in developing the necessary recursion. 

\begin{lem}
\label{P3cases}
	Fix bifix-free word $L$. 
	Let $W = LAL$ be a $Z_2$-instance with a $Z_2$-bifix. Then $LAL$ can be written in exactly one of the following ways:
	\begin{enumerate}[$\<$i$\>$]
		\item $LAL = LBLCLBL$ with $LBL$ the shortest $Z_2$-bifix of $W$ and $|C|>0$;
		\item $LAL = LBLLBL$ with $LBL$ the shortest $Z_2$-bifix of $W$;
		\item $LAL = LBLBL$ with $LBL$ the shortest $Z_2$-bifix of $W$;
		\item $LAL = LLFLLFLL$ with $LLFLL$ the shortest $Z_2$-bifix of $W$;
		\item $LAL = LLLL$.
	\end{enumerate}
\end{lem}

\begin{proof}
	With some thought, the reader should recognize that the five listed cases are in fact mutually exclusive. 
The proof that these are the only possibilities follows.

	Given that $W$ has a $Z_2$-bifix and $L$ is bifix-free, it follows that $W$ has a $Z_2$-bifix $LBL$ for some nonempty $B$. 
Let $LBL$ be chosen of minimal length. We break this proof into nine cases depending on the lengths of $L$ and $LBL$ (Figure~\ref{overlap}). 
Set $m = |W|$, $\ell = |L|$, and $k = |LBL|$.

	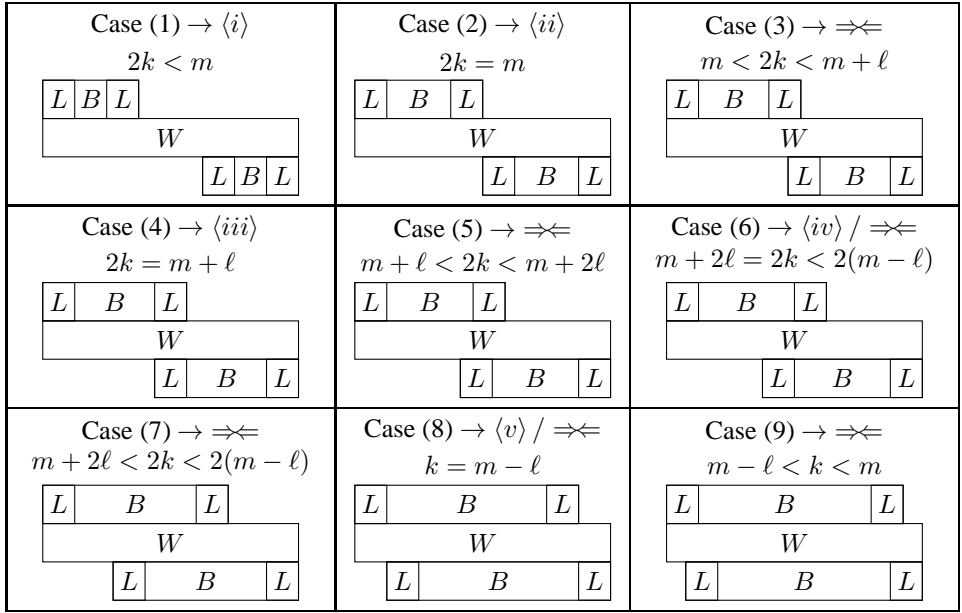
\begin{figure}[ht] 
	\centering

		\begin{tabular}{| c | c | c |} \hline

		\begin{tikzpicture}[scale=.85]
			\draw (0,0) rectangle node{$W$} (4,.6);
			\draw (2,1.7) node[above]{Case (1) $\rightarrow \ang{i}$};
			\draw (2,1.2) node[above]{$2k < m$};
		
			\draw (0,1.2) rectangle node{$B$} (1.5,.6);
			\draw (0,.6) rectangle node{$L$} (.5,1.2);
			\draw (1,1.2) rectangle node{$L$} (1.5,.6);
		
			\draw (4,0) rectangle node{$B$} (2.5,-.6);
			\draw (3,-.6) rectangle node{$L$} (2.5,0);
			\draw (4,0) rectangle node{$L$} (3.5,-.6);
		
		\end{tikzpicture}
& 
		\begin{tikzpicture}[scale=.85]
				\draw (0,0) rectangle node{$W$} (4,.6);
				\draw (2,1.7) node[above]{Case (2) $\rightarrow \ang{ii}$};
				\draw (2,1.2) node[above]{$2k = m$};
			
				\draw (0,1.2) rectangle node{$B$} (2,.6);
				\draw (0,.6) rectangle node{$L$} (.5,1.2);
				\draw (2,1.2) rectangle node{$L$} (1.5,.6);
			
				\draw (4,0) rectangle node{$B$} (2,-.6);
				\draw (2,-.6) rectangle node{$L$} (2.5,0);
				\draw (4,0) rectangle node{$L$} (3.5,-.6);
		\end{tikzpicture}
&		
		\begin{tikzpicture}[scale=.85]
				\draw (0,0) rectangle node{$W$} (4,.6);
				\draw (2,1.7) node[above]{Case (3) $\rightarrow \; \Rightarrow\!\Leftarrow$};
				\draw (2,1.2) node[above]{$m < 2k < m + \ell$};
			
				\draw (0,1.2) rectangle node{$B$} (2.1,.6);
				\draw (0,.6) rectangle node{$L$} (.5,1.2);
				\draw (2.1,1.2) rectangle node{$L$} (1.6,.6);
			
				\draw (4,0) rectangle node{$B$} (1.9,-.6);
				\draw (1.9,-.6) rectangle node{$L$} (2.4,0);
				\draw (4,0) rectangle node{$L$} (3.5,-.6);
		\end{tikzpicture}
\\ \hline
		\begin{tikzpicture}[scale=.85]
				\draw (0,0) rectangle node{$W$} (4,.6);
				\draw (2,1.7) node[above]{Case (4) $\rightarrow \ang{iii}$};
				\draw (2,1.2) node[above]{$2k = m + \ell$};
			
				\draw (0,1.2) rectangle node{$B$} (2.25,.6);
				\draw (0,.6) rectangle node{$L$} (.5,1.2);
				\draw (2.25,1.2) rectangle node{$L$} (1.75,.6);
			
				\draw (4,0) rectangle node{$B$} (1.75,-.6);
				\draw (1.75,-.6) rectangle node{$L$} (2.25,0);
				\draw (4,0) rectangle node{$L$} (3.5,-.6);
		\end{tikzpicture}
&		
		\begin{tikzpicture}[scale=.85]
				\draw (0,0) rectangle node{$W$} (4,.6);
				\draw (2,1.7) node[above]{Case (5) $\rightarrow \; \Rightarrow\!\Leftarrow$};
				\draw (2,1.2) node[above]{$m + \ell < 2k < m + 2\ell$};
			
				\draw (0,1.2) rectangle node{$B$} (2.35,.6);
				\draw (0,.6) rectangle node{$L$} (.5,1.2);
				\draw (2.35,1.2) rectangle node{$L$} (1.85,.6);
			
				\draw (4,0) rectangle node{$B$} (1.65,-.6);
				\draw (1.65,-.6) rectangle node{$L$} (2.15,0);
				\draw (4,0) rectangle node{$L$} (3.5,-.6);
		\end{tikzpicture}
&		
		\begin{tikzpicture}[scale=.85]
				\draw (0,0) rectangle node{$W$} (4,.6);
				\draw (2,1.7) node[above]{Case (6) $\rightarrow \ang{iv}/\Rightarrow\!\Leftarrow$};
				\draw (2,1.2) node[above]{$m + 2\ell = 2k < 2(m - \ell)$};
			
				\draw (0,1.2) rectangle node{$B$} (2.5,.6);
				\draw (0,.6) rectangle node{$L$} (.5,1.2);
				\draw (2.5,1.2) rectangle node{$L$} (2,.6);
			
				\draw (4,0) rectangle node{$B$} (1.5,-.6);
				\draw (1.5,-.6) rectangle node{$L$} (2,0);
				\draw (4,0) rectangle node{$L$} (3.5,-.6);
		\end{tikzpicture}
\\ \hline
		\begin{tikzpicture}[scale=.85]
				\draw (0,0) rectangle node{$W$} (4,.6);
				\draw (2,1.7) node[above]{Case (7) $\rightarrow \; \Rightarrow\!\Leftarrow$};
				\draw (2,1.2) node[above]{$m + 2\ell < 2k < 2(m - \ell)$};
			
				\draw (0,1.2) rectangle node{$B$} (2.9,.6);
				\draw (0,.6) rectangle node{$L$} (.5,1.2);
				\draw (2.9,1.2) rectangle node{$L$} (2.4,.6);
			
				\draw (4,0) rectangle node{$B$} (1.1,-.6);
				\draw (1.1,-.6) rectangle node{$L$} (1.6,0);
				\draw (4,0) rectangle node{$L$} (3.5,-.6);
		\end{tikzpicture}
&		
		\begin{tikzpicture}[scale=.85]
				\draw (0,0) rectangle node{$W$} (4,.6);
				\draw (2,1.7) node[above]{Case (8) $\rightarrow \ang{v} / \Rightarrow\!\Leftarrow$};
				\draw (2,1.2) node[above]{$k = m - \ell$};
			
				\draw (0,1.2) rectangle node{$B$} (3.5,.6);
				\draw (0,.6) rectangle node{$L$} (.5,1.2);
				\draw (3,1.2) rectangle node{$L$} (3.5,.6);
			
				\draw (4,0) rectangle node{$B$} (.5,-.6);
				\draw (.5,-.6) rectangle node{$L$} (1,0);
				\draw (4,0) rectangle node{$L$} (3.5,-.6);
		\end{tikzpicture}
&		
		\begin{tikzpicture}[scale=.85]
				\draw (0,0) rectangle node{$W$} (4,.6);
				\draw (2,1.7) node[above]{Case (9) $\rightarrow \; \Rightarrow\!\Leftarrow$};
				\draw (2,1.2) node[above]{$m - \ell < k < m$};
			
				\draw (0,1.2) rectangle node{$B$} (3.7,.6);
				\draw (0,.6) rectangle node{$L$} (.5,1.2);
				\draw (3.7,1.2) rectangle node{$L$} (3.2,.6);
			
				\draw (4,0) rectangle node{$B$} (.3,-.6);
				\draw (.3,-.6) rectangle node{$L$} (.8,0);
				\draw (4,0) rectangle node{$L$} (3.5,-.6);
		\end{tikzpicture}
\\ \hline
		\end{tabular}

		\caption[All possible ways the minimal $Z_2$-bifix of a word can overlap.]{All possible ways the minimal $Z_2$-bifix of $W$ can overlap, with $m = |W|$, $\ell = |L|$, and $k = |LBL|$} \label{overlap}

	\end{figure}

	\begin{enumerate}[\text{Case} (1):]
		\item $2k < m$. This is $\< i\>$.
		\item $2k = m$. This is $\<ii\>$.
		\item $m < 2k < m + \ell$. In $LAL$, the first and last occurrences of $LBL$ overlap by a length strictly between $0$ and $\ell$. This is impossible, since $L$ is bifix-free.
		\item $2k = m + \ell$. This is $\<iii\>$
		\item $m + \ell < 2k < m + 2\ell$. The first and last occurrences of $LBL$ overlap by a length strictly between $\ell$ and $2\ell$. This is impossible, since $L$ is bifix-free.
		\item $m + 2\ell = 2k < 2(m - \ell)$. $LAL = L(DL)(LE)L$ where $DL = B = LE$. Thus $L$ is a bifix of $B$, so $LAL = LLFLLFLL$ where $B = LFL$. If $|F|>0$, this is $\<iv\>$. If $|F|=0$, then $LAL = LLLLLL$. But this contradicts the minimality of $LBL$, since $LLLLLL$ has $Z_2$-bifix $LLL$, which is shorter than $LBL = LLLL$.
		\item $m + 2\ell < 2k < 2(m - \ell)$. $LAL = LDLELD'L$ where $DLE = B = ELD'$. 
		Since $EL$ is a prefix of $B$, $LEL$ is a prefix of $LAL$. 
		Likewise, since $LE$ is a suffix of $B$, $LEL$ is a suffix of $LAL$.
		Therefore, $LEL$ is a bifix of $LAL$ and $|LEL| < |LDLEL| = |LBL|$, contradicting the minimality of $LBL$.
		\item $k = m - \ell$. $LAL = LLCLL$ where $LC = B = CL$. If $|C|=0$, this is $\<v\>$. Otherwise, $LCL$ is a bifix of $LAL$, contradicting the minimality of $LBL$.
		\item $m - \ell < k < m$. The first and last occurrences of $LBL$ overlap by a length strictly between $k-\ell$ and $k$. This is impossible, since $L$ is bifix-free.
	\end{enumerate}
\end{proof}

For fixed bifix-free word $L$ of length $\ell$, define $b_m^\ell$ to count the number of $Z_2$ words with bifix $L$ that are $Z_2$-bifix-free $q$-ary words of length $m$. 
Then 
\begin{equation} \label{eqn:IZ3}
\II(Z_3,q) = \sum_{\ell = 1}^{\infty} \left( a_\ell \sum_{m = 1}^\infty b_m^\ell q^{-2m} \right).
\end{equation}

In order to form a recursive definition of $b_n$ as we did for $a_n$, we now describe two new terms. Let $AB$ be a word of length $W$ with $|A| = \ceil{W/2}$ and $|B| = \floor{W/2}$. Then $AB$ has $q$ length-$(n+1)$ \textit{children} of the form $AxB$, each having $AB$ as its \textit{parent}. In this way every nonempty word has exactly $q$ children and exactly 1 parent, which establishes the 1:$q$ ratio of words of length $n$ to words of length $n+1$. The set of a word's children together with successive generations of progeny we refer to as that word's \textit{descendants}.

\begin{thm} \label{P3}
	$b_n^\ell = c_n^\ell + d_n^\ell$ where $c_n=c_n^\ell$ and $d_n=d_n^\ell$ are defined recursively as follows:
	\begin{eqnarray*}
		\text{For even }\ell:\\
		c_1 = \cdots = c_{2\ell} & = & 0,\\
		c_{2\ell+1} & = & q,\\
		c_{4\ell} & = & qc_{4\ell-1} - (c_{5\ell/2} + 1),\\
		c_{5\ell} & = & qc_{5\ell - 1} - (c_{5\ell/2} + c_{3\ell} - 1),\\
		c_{5\ell+1} & = & q(c_{5\ell} + c_{3\ell} - 1),\\
		c_{6\ell} & = & qc_{6\ell-1} - (c_{3\ell} - 1 + c_{5\ell/2});\\
		c_{2k} & = & qc_{2k-1} - (c_k + c_{k + \ell/2}) \text{ for } k>\ell,k \not\in\{2\ell,5\ell/2,3\ell\},\\
		c_{2k+1} & = & q(c_{2k} + c_{k + \ell/2}) \text{ for } k>\ell, k \neq 5\ell/2,\\
		d_1 = \cdots = d_{4\ell} & = & 0,\\
		d_{4\ell+1} & = & q,\\
		d_{5\ell} & = & qd_{5\ell-1} - 1,\\
		d_{5\ell+1} & = & q(d_{5\ell} + 1),\\
		d_{6\ell} & = & qd_{6\ell-1} - 1,\\
		d_{2k} & = & qd_{2k-1} - (d_k + d_{k+\ell} + d_{k + \ell/2}) \text{ for } k>2\ell,k \not\in \{5\ell/2,3\ell\},\\
		d_{2k+1} & = & q(d_{2k} + d_{k+\ell} + d_{k + \ell/2}) \text{ for } k\geq 2\ell, k \neq 5\ell/2.\\
		\text{For odd }\ell>1:\\
		c_1 = \cdots = c_{2\ell} & = & 0,\\
		c_{2\ell+1} & = & q,\\
		c_{4\ell} & = & q\left(c_{4\ell-1} + c_{\floor{\frac{5\ell}{2}}}\right) - (c_{2\ell} +1),\\
		c_{5\ell} & = & qc_{5\ell - 1} - (c_{3\ell} - 1),\\
		c_{5\ell+1} & = & q(c_{5\ell} + c_{3\ell} - 1) - c_{\ceil{\frac{5\ell}{2}}},\\
		c_{6\ell} & = & q\left(c_{6\ell-1} + c_{\floor{\frac{7\ell}{2}}}\right) - (c_{3\ell} -1),\\
		c_{2k} & = & q\left(c_{2k-1} + c_{k+\floor{\frac{\ell}{2}}}\right) - c_k; k>\ell,k \not\in \left\{2\ell,\ceil{\frac{\ell}{2}},3\ell\right\},\\
		c_{2k+1} & = & qc_{2k} - c_{k+\ceil{\frac{\ell}{2}}}; k>\ell,k \neq \floor{\frac{5\ell}{2}};\\
	\end{eqnarray*}
\\
	\begin{eqnarray*}
		d_1 = \cdots = d_{4\ell} & = & 0,\\
		d_{4\ell+1} & = & q,\\
		d_{5\ell} & = & qd_{5\ell-1} - 1,\\
		d_{5\ell+1} & = & q(d_{5\ell} + 1),\\
		d_{6\ell} & = & qd_{6\ell-1} - 1,\\
		d_{2k} & = & q\left(d_{2k-1}+ d_{k + \floor{\frac{\ell}{2}}}\right) - (d_k + d_{k+\ell}); k>2\ell,k\not\in \left\{\ceil{\frac{5\ell}{2}},3\ell\right\},\\
		d_{2k+1} & = & q\left(d_{2k} + d_{k+\ell}\right) -  d_{k + \ceil{\frac{\ell}{2}}}; k> 2\ell, k  \neq \floor{\frac{5\ell}{2}}.\\
		\text{For }\ell=1:\\
		c_1 = c_1 = c_2 & = & 0,\\
		c_3 & = & q,\\
		c_4 & = & qc_3 - 1,\\
		c_5 & = & qc_4 - (c_3 - 1),\\
		c_{6} & = & q(c_5 + c_3 - 1) - (c_3 - 1),\\
		c_{2k} & = & q(c_{2k-1} + c_k) - c_k; k>3,\\
		c_{2k+1} & = & qc_{2k} - c_{k+1}; k>2;\\
		d_1 = d_2 = d_3 = d_4 & = & 0,\\
		d_5 & = & q - 1,\\
		d_{6} & = & q(d_5+1) - 1,\\
		d_{2k} & = & q(d_{2k-1}+ d_k) - (d_k + d_{k+1}); k>3,\\
		d_{2k+1} & = & q(d_{2k} + d_{k+1}) -  d_{k + 1}; k> 2.
	\end{eqnarray*}

\end{thm}

\begin{proof}
	Fix a bifix-free word $L$ of length $\ell$.
	The full recursion is too messy to prove all at once, so we build up to it in stages.
	Within each stage, $\approx$ indicates an incomplete definition.
	Example word trees with small $q$ and short $L$ are found in Appendix~\ref{TREES}.

	\textbf{Stage I}\\
	Since $L$ is bifix free, any $Z_2$-instance with $L$ as a bifix has to be of greater length than $2\ell.$ Thus we have $b_1 = \cdots = b_{2\ell} = 0$. 
	The only such words of length $2\ell+1$ are of the form $LxL$ for some letter $x$, therefore, $b_{2\ell+1} = q$.

	Every word of length $n>2\ell+1$ has $L$ as a bifix if and only if its parent has $L$ as a bifix. 
	This is why, for $k>\ell$, the definition of $b_{2k}$ includes the term $qb_{2k-1}$, and the definition of $b_{2k+1}$ includes the term $qb_{2k}$. If $b_n$ were counting $Z_2$-instances with bifix $L$, we would be done. 
	However, we do not want $b_n$ to count words that have a $Z_2$-bifix. 
	Thus, we must deal with each of the 5 cases listed in Lemma~\ref{P3cases}.

	First, let us deal with case $\ang{ii}$: $LAL = LBLLBL$ with $LBL$ the shortest $Z_2$-bifix of $LAL$. 
	The number of these of length $2k$, with $k > \ell$, is $b_k$. 
	Therefore, in the definition of $b_{2k}$, we subtract $b_k$. 
	Conveniently, the descendants of case-$\ang{ii}$ words are precisely words of case $\ang{i}$. 
	Therefore, we have accounted for two cases at once.

	Next, let us look at case $\ang{iii}$: $LAL = LBLBL$ with $LBL$ the shortest $Z_2$-bifix of $LAL$. 
	For the moment, assume $|L| = \ell$ is even. Then $|LBLBL|$ is even. 
	The number of such words of length $2k$, with $k > \ell$, is $b_{k+\ell/2}$. 
	We want to exclude words of this form, but we do not necessarily want to exclude their children. 
	Therefore, in the definition of $b_{2k}$ we subtract $b_{k+\ell/2}$, but then we add $qb_{k+\ell/2}$ in the definition of $b_{2k+1}$. 

	Now we look at when $|L|$ is odd, so $|LBLBL|$ is odd. 
	The number of such words of length $2k+1$, with $k > \ell$, is $b_{k+\ceil{\ell/2}}$. 
	Therefore, in the definition of $b_{2k+1}$ we subtract $b_{k+\ceil{\ell/2}}$, but then we add $qb_{(k - 1) +\ceil{\ell/2}} = qb_{k + \floor{\ell/2}}$ in the definition of $b_{(2(k-1)+1) + 1} = b_{2k}$.

	Our work so far renders the following tentative definition of $b_n$.
	\begin{eqnarray*}
		\text{For even }\ell:\\
		b_1 = \cdots = b_{2\ell} & = & 0,\\
		b_{2\ell+1} & = & q,\\
		b_{2k} & \approx & qb_{2k-1} - (b_k + b_{k + \ell/2}) \text{ for } k>\ell,\\
		b_{2k+1} & \approx & q(b_{2k} + b_{k + \ell/2}) \text{ for } k>\ell.\\
		\text{For odd }\ell:\\
		b_1 = \cdots = b_{2\ell} & = & 0,\\
		b_{2\ell+1} & = & q,\\
		b_{2k} & \approx & q(b_{2k-1} + b_{k+\floor{\ell/2}}) - b_k \text{ for } k>\ell,\\
		b_{2k+1} & \approx & qb_{2k} - b_{k+\lceil\ell/2\rceil} \text{ for } k>\ell.
	\end{eqnarray*}

	We continue with case $\ang{iv}$: $LAL = LLFLLFLL$ with $LLFLL$ the shortest $Z_2$-bifix of $LAL$. 
	Note that $|LLFLLFLL|$ is even. 
	It would apear that the number of such words of length $2k$ would be $b_{k - \ell}$ (counting words of the form $LFL$), which we could deal with in the same fashion as we did for case $\ang{iii}$. 
	However, when counting words of the form $LFL$, we do not want words of the form $LLGLL$, because $LLFLLFLL = LLLGLLLLGLLL$ is already accounted for in case $\ang{i}$.

	\textbf{Stage II}\\
	To address this issue, we will define two different recursions. 
	Let $d_n$ count the $Z_2$-instances of the form $LLALL$ that are $Z_2$-bifix free. 
	Let $c_n$ count all other $Z_2$-instances of the form $LAL$ that are $Z_2$-bifix free. Therefore, $b_n = c_n + d_n$ by definition. 

	As with $b_n$, we quickly see that $c_n = 0$ for $n\leq 2\ell$ and $c_{2\ell+1} = q$. 
	Now the shortest words counted by $d_n$ are of the form $LLxLL$ for some letter $x$, so $d_n = 0$ for $n\leq 4\ell$ and $d_{4\ell+1} = q$. 

	To deal with cases $\ang{i}$ and $\ang{ii}$, we can do the same things as before, but recognizing that $LL$ is a bifix of $LBLLBL$ if and only if $LL$ is a bifix of $LBL$. 
	Therefore, subtract $c_k$ in the definition of $c_{2k}$ and subtract $d_k$ in the definition of $d_{2k}$ (both for $k>\ell$).

	We also deal with case $\ang{iii}$ as before, recognizing that $LL$ is a bifix of $LBLBL$ if and only if $LL$ is a bifix of $LBL$. 
	For even $\ell$: subtract $c_{k+\ell/2}$ in the definition of $c_{2k}$ and add $qc_{k+\ell/2}$ in the definition of $c_{2k+1}$; subtract $d_{k+\ell/2}$ in the definition of $d_{2k}$ and add $qd_{k+\ell/2}$ in the definition of $d_{2k+1}$. 
	For odd $\ell$: subtract $c_{k+\ceil{\ell/2}}$ in the definition of $c_{2k+1}$ and add $qc_{k+\floor{\ell/2}}$ in the definition of $c_{2k}$; subtract $d_{k+\ceil{\ell/2}}$ in the definition of $d_{2k+1}$ and add $qd_{k+\floor{\ell/2}}$ in the definition of $d_{2k}$.

	Having split $b_n$ into $c_n$ and $d_n$, we can address case $\ang{iv}$: $LAL = LLFLLFLL$ with $LLFLL$ the shortest $Z_2$-bifix of $LAL$. 
	These words are counted by $d_n$, not by $c_n$, and there are $d_{k+\ell}$ such words of length $2k$. 
	Therefore, we subtract $d_{k+\ell}$ in the definition of $d_{2k}$ and add $qd_{k+\ell}$ in the definition of $d_{2k+1}$.

	This brings us to the following tentative definitions of $c_n$ and $d_n$.
	\begin{eqnarray*}
		\text{For even }\ell:\\
		c_1 = \cdots = c_{2\ell} & = & 0,\\
		c_{2\ell+1} & = & q,\\
		c_{2k} & \approx & qc_{2k-1} - (c_k + c_{k + \ell/2}),\\
		c_{2k+1} & \approx & q(c_{2k} + c_{k + \ell/2});\\
		d_1 = \cdots = d_{4\ell} & = & 0,\\
		d_{4\ell+1} & = & q,\\
		d_{2k} & \approx & qd_{2k-1} - (d_k + d_{k+\ell} + d_{k + \ell/2}),\\
		d_{2k+1} & \approx & q(d_{2k} + d_{k+\ell} + d_{k + \ell/2}).\\
		\text{For odd }\ell:\\
		c_1 = \cdots = c_{2\ell} & = & 0,\\
		c_{2\ell+1} & = & q,\\
		c_{2k} & \approx & q(c_{2k-1} + c_{k+\floor{\ell/2}}) - c_k,\\
		c_{2k+1} & \approx & qc_{2k} - c_{k+\lceil\ell/2\rceil};\\
		d_1 = \cdots = d_{4\ell} & = & 0,\\
		d_{4\ell+1} & \approx & q,\\
		d_{2k} & \approx & q(d_{2k-1}+ d_{k + \lfloor\ell/2\rfloor}) - (d_k + d_{k+\ell}),\\
		d_{2k+1} & \approx & q(d_{2k} + d_{k+\ell}) -  d_{k + \lceil\ell/2\rceil}.
	\end{eqnarray*}

	\textbf{Stage III}\\
	Next, let us deal with case $\ang{v}$: $LLLL$. 
	We merely need to subtract 1 in the definition of $c_{4\ell}$. 
	Since all of the words counted by $d_n$ are descendants of $LLLL$, this is what prevents overlap of the words counted by $c_n$ and $d_n$.

	There was a small omission in the previous stage. 
	When dealing with cases $\ang{i}$ and $\ang{ii}$, we pointed out that $LL$ is a bifix of $LBLLBL$ if and only if $LL$ is a bifix of $LBL$, this was a true and important observation. 
	The one problem is that $LLL$ has $LL$ as a bifix but is not of the form $LLALL$. 
	Therefore, $LLLLLL$ was ``removed'' in the definition of $c_{6\ell}$ when it should have been ``removed'' from $d_{6\ell}$. 
	We must account for this by adding 1 in the definition of $c_{6\ell}$ and subtracting 1 in the definition of $d_{6\ell}$.

	Similarly, in dealing with case $\ang{iii}$, we ``removed'' $LLLLL$ in the definition of $c_{5\ell}$ and ``replaced'' its children in the definition of $c_{5\ell+1}$. These should have happened to $d_n$. 
	Therefore, we add 1 and subtract $q$ in the definitions of $c_{5\ell}$ and $c_{5\ell+1}$, respectively, then subtract 1 and add $q$ in the definitions of $d_{5\ell}$ and $d_{5\ell+1}$, respectively.

	Since $LLL$ does not cause any trouble with case $\ang{iv}$, we are done building the recursive definition for even $\ell$ as found in the theorem statement.

	\textbf{Stage IV}\\
	The recursion for odd $\ell$ has the additional caveat that $\ell \neq 1$. When $\ell=1$, there exist conflicts in the recursive definitions: $4\ell + 1 = 5\ell$ and $5\ell+1 = 6\ell$. After consolidating the``adjustments'' for these cases, we get the definition for $\ell=1$ as appears in the theorem statement.
\end{proof}

With our recursively defined sequences $a_n$ and $b_n$, the latter in terms of $c_n$ and $d_n$, we are now able to formulate Theorem~\ref{thm:IZ2} for $Z_3$.

\begin{thm} \label{thm:IZ3}
For integers $q \geq 2$,
\begin{eqnarray*}
	 \II(Z_3,q) & = & \sum_{\ell = 1}^{\infty} a_\ell \left(\sum_{i = 0}^{\infty}(G(i) + H(i))\right).
\end{eqnarray*} 

where
\begin{eqnarray*}
	G(i) = G_\ell^{(q)}(i) &=& \frac{ (-1)^i r\!\left(q^{-2^{i+1}}\right) \prod_{j = 0}^{i-1} s\!\left(q^{-2^{j+1}}\right)}{ \prod_{k = 0}^i \left(1 - q^{1-2^{k+1}}\right)};\\
	r(x) = r_\ell^{(q)}(x) & = & qx^{2\ell+1} - x^{4\ell} + x^{5\ell} - qx^{5\ell+1} +x^{6\ell};\\
	s(x) = s_\ell^{(q)}(x) & = & 1 - qx^{1-\ell} + x^{-\ell};\\
	H(i) = H_\ell^{(q)}(i) &=& \frac{ (-1)^i u\!\left(q^{-2^{i+1}}\right) \prod_{j = 0}^{i-1} v\!\left(q^{-2^{j+1}}\right) }{ \prod_{k = 0}^i \left(1 - q^{1-2^{k+1}}\right)};\\
	u(x) = u_\ell^{(q)}(x) & = & qx^{4\ell+1} - x^{5\ell} + qx^{5\ell+1} - x^{6\ell};\\
	v(x) = v_\ell^{(q)}(x) & = & 1 - qx^{1-\ell} + x^{-\ell} - qx^{1-2\ell} + x^{-2\ell}.
\end{eqnarray*}
\end{thm}

\begin{proof}
	Recalling Equation~\eqref{eqn:IZ3},
	\begin{eqnarray*}
		\II(Z_3,q) & = & \sum_{\ell = 1}^{\infty} \left( a_\ell \sum_{m = 1}^\infty b_m^\ell q^{-2m} \right) \\
		& = & \sum_{\ell = 1}^{\infty} \left( a_\ell \sum_{m = 1}^\infty \left(c_m^\ell + d_m^\ell\right) q^{-2m} \right).
	\end{eqnarray*}

Similar to our proof for $\II(Z_2,q)$, let us define generating functions for the sequences $c_n = c_n^\ell$ and $d_n = d_n^\ell$:
\[ g(x) = g_\ell^{(q)}(x) = \sum_{i = 1}^\infty c_nx^n \text{ and } h(x) = h_\ell^{(q)}(x) = \sum_{i = 1}^\infty d_nx^n. \]
Despite having to write the recursive relations three different ways, depending on $\ell$, the underlying recursion is fundamentally the same and results in the following functional equations:
\begin{eqnarray}
	g(x) & = & q\left(xg(x) + x^{1-\ell}g(x^2) + x^{2\ell+1} - x^{5\ell+1}\right) \label{eqn:g}\\
	\nonumber & & - \left(g(x^2) + x^{-\ell}g(x^2) + x^{4\ell} - x^{5\ell} - x^{6\ell}\right);\\
	h(x) & = & q\left(xh(x) + x^{1-2\ell}h(x^2) + x^{1-\ell}h(x^2) + x^{4\ell+1} + x^{5\ell+1}\right) \label{eqn:h} \\
	\nonumber & & - \left(h(x^2) + x^{-2\ell}h(x^2) + x^{-\ell}h(x^2) + x^{5\ell} + x^{6\ell}\right).
\end{eqnarray}

Solving \eqref{eqn:g} for $g(x)$, we get
\begin{eqnarray}
	g(x) = \frac{r(x) - s(x)g(x^2)}{1-qx}, \label{eqn:g2}
\end{eqnarray}
with $r(x)$ and $s(x)$ as defined in the theorem statement. 
Expanding \eqref{eqn:g2} gives
\begin{eqnarray}
	\nonumber g(x) & = & \frac{r(x) - s(x)g(x^2)}{1-qx} \\
	\nonumber & = & \frac{r(x)}{1-qx}\left(1 -  \frac{s(x)}{r(x)}g(x^2)\right)\\
	\nonumber & = & \frac{r(x)}{1-qx}\left(1 -  \frac{s(x)}{r(x)}\frac{r(x^2) - s(x^2)g(x^4)}{1-qx^2}\right)\\
	\nonumber & = & \frac{r(x)}{1-qx}\left(1 -  \frac{s(x)}{r(x)}\frac{r(x^2)}{1-qx^2}\left(1 - \frac{s(x^2)}{r(x^2)}g(x^4) \right)\right)\\
	\nonumber & \vdots & \\
	& = & \sum_{i = 0}^\infty \frac{ (-1)^i r\!\left(x^{2^i}\right) \prod_{j = 0}^{i-1} s\!\left(x^{2^j}\right) }{ \prod_{k = 0}^i \left(1 - qx^{2^k}\right)}.
\end{eqnarray}
Likewise, solving \eqref{eqn:h} for $h(x)$, we get
\begin{eqnarray}
	h(x) &=& \frac{u(x) - v(x)h(x^2)}{1-qx}\\
	& = & \sum_{i = 0}^\infty \frac{(-1)^i u\!\left(x^{2^i}\right) \prod_{j = 0}^{i-1} v\!\left(x^{2^j}\right) }{ \prod_{k = 0}^i \left(1 - qx^{2^k}\right)},
\end{eqnarray}
with $u(x)$ and $v(x)$ as defined in the theorem statement.
\end{proof}

\begin{cor} \label{P3bounds}
For integers $N \geq 0$ and $M \geq 0$,
\begin{eqnarray*}
	\sum_{\ell = 1}^{N} a_\ell \left(\sum_{i = 0}^{2M+1} (G(i) + H(i)) \right) & \leq & \II(Z_3,q);\\
	 \II(Z_3,q) & \leq & q^{-N} + \sum_{\ell = 1}^{N} a_\ell \left(\sum_{i = 0}^{2M}(G(i) + H(i))\right),
\end{eqnarray*} 
with $G(i) = G_\ell^{(q)}(i)$ and $H(i) = H_\ell^{(q)}(i)$ as defined in Theorem~\ref{thm:IZ3}.
\end{cor}

\begin{proof} 
	For fixed integers $q\geq 2$ and $\ell \geq 1$, $\sum_{i=0}^\infty (G(i) + H(i))$ is an alternating series. 
	We need to show that the sequence $|G(i)+H(i)|$ is decreasing. Since $(-1)^iG(i)>0$ and $(-1)^iH(i) > 0$ for each $i$, $|G(i) +H(i)| = |G(i)| + |H(i)|$. 
	Thus it suffices to show that $\left\{|G(i)|\right\}_{i=1}^{\infty}$ and $\left\{|H(i)|\right\}_{i = 1}^\infty$ are both decreasing sequences, the routine proof of which can be found in the appendices (Lemma~\ref{lemGH}).
	
	Now for any integer $M\geq 0$:
	\[\sum_{i=0}^{2M+1} G_\ell(i) + H_\ell(i) < \sum_{m = 0}^{\infty} b_m^\ell q^{-2m} < \sum_{i=0}^{2M} G_\ell(i) + H_\ell(i).\]
	Moreover, since the $a_\ell$ are nonnegative, the lower bound for the theorem is evident. 
	For a bifix-free word $L$ of length $\ell$, $\sum_{m = 0}^{\infty} b_m^\ell q^{-2m}$ is the limit, as $M \rightarrow \infty$, of the probability that a word of length $M$ is a $Z_3$-instance of the form $LALBLAL$. 
	A necessary condition for such a word is that it starts and ends with $L$, which (for $M\geq 2\ell$) has probability $q^{-2\ell}$. 
	Also $a_\ell$ counts the number of bifix-free words of length $\ell$, so $a_\ell \leq q^\ell$. 
	Hence for any integer $N \geq 0$:
	\begin{eqnarray*}
		\II(Z_3,q) & < & \sum_{\ell=1}^{N}a_\ell \sum_{m = 0}^{\infty} b_m^\ell q^{-2m} + \sum_{\ell = N+1}^{\infty}q^\ell \left(q^{-2\ell}\right)\\
		& = &\sum_{\ell=1}^{N}a_\ell \sum_{m = 0}^{\infty} b_m^\ell q^{-2m} + \sum_{\ell = N+1}^{\infty} q^{-\ell}\\
		& \leq &\sum_{\ell=1}^{N}a_\ell \sum_{m = 0}^{\infty} b_m^\ell q^{-2m} + q^{-N}.
	\end{eqnarray*}
\end{proof}

\begin{table}[ht]
\centering
	\caption{Approximate values of $\II(Z_3,q)$ for $2 \leq q \leq 6$.} \label{table:IZ3}

	\begin{tabular}{c | c | c | c | c | c}
		$q$ & 2 & 3 & 4 & 5 & 6 \\ \hline
		$\II(Z_3,q)$ & 0.11944370 & 0.01835140 & 0.00519251 & 0.00199739 & 0.00092532
	\end{tabular}

\end{table}

The values in Table~\ref{table:IZ3} were generated by the Sage code found in Appendix~\ref{P3code}, which was derived directly from Corollary~\ref{P3bounds} and can be used to compute $\II(Z_3,q)$ to arbitrary precision for any $q\geq 2$.

\section{Bounding \texorpdfstring{$\II(Z_n,q)$ for Arbitrary $n$}{the asymptotic instance probability of Zn}} \label{IZn}

This programme is not practical for $n$ in general. 
The number of cases for a generalization of Lemma 3.1 is likely to grow with $n$. 
Even if that stabilizes somehow, the expression for calculating $\II(Z_n,q)$ requires $n$ nested infinite series.
Nevertheless, ignoring some of the more subtle details, we proceed with this method to obtain computable upper bounds for $\II(Z_n,q)$.



Fix a $Z_{n-2}$-instance $L$ of length $\ell \geq 1$, let $\hat{b}_m^\ell$ be the number of words of length $m$ of the form $LAL$ for $A \neq \varepsilon$ but not of the form $LBLBL$, $LBLLBL$, or $LBLCLBL$. 
This corresponds to Stage~I from the proof of Theorem~\ref{P3}. 
As we do not account for the structure of $L$, $\hat{b}$ is an overcount for the number of $Z_{n-1}$-instances of the form $LAL$ that do not have a $Z_{n-1}$-bifix of the form $LAL$. 
Then $\hat{b}_m = \hat{b}_m^\ell$ is recursively defined as follows:

\begin{eqnarray*}
	\text{For even } \ell: \\
	\hat{b}_0 = \cdots = \hat{b}_{2\ell} & = & 0,\\
	\hat{b}_{2k} & = & q\hat{b}_{2k-1} - (\hat{b}_k + \hat{b}_{k + \ell/2}) \text{ for } k>\ell, \\
	\hat{b}_{2k+1} & = & q(\hat{b}_{2k} + \hat{b}_{k+\ell/2}) \text{ for } k>\ell. \\
	\text{For odd } \ell: \\
	\hat{b}_0 = \cdots = \hat{b}_{2\ell} & = & 0, \\
	\hat{b}_{2k} & = & q(\hat{b}_{2k-1} + \hat{b}_{k + \floor{\ell/2}}) - \hat{b}_k \text{ for } k>\ell, \\
	\hat{b}_{2k+1} & = & q\hat{b}_{2k} - \hat{b}_{k+\ceil{\ell/2}} \text{ for } k>\ell. 
\end{eqnarray*}

The associated generating function $\hat{f}_\ell(x) := \hat{f}_\ell^{(q)}(x) = \sum_{m=1}^{\infty} \hat{b}_m^\ell x^m$ satisfies
\[\hat{f}_\ell(x) = q(x^{2\ell+1} + x\hat{f}(x) + x^{1-\ell}\hat{f}(x^2)) - (\hat{f}(x^2) + x^{-\ell}\hat{f}(x^2)).\]
Therefore, setting $t_\ell(x) = t_\ell^{(q)}(x) = 1 - qx^{1-\ell} + x^{-\ell}$,
\begin{eqnarray*}
	\hat{f}_\ell(x) & = & \frac{qx^{2\ell+1} - t_\ell(x)\hat{f}(x^2)}{1-qx} \\
	& = &  q\cdot \sum_{i = 0}^\infty \frac{ (-1)^i x^{(2^i)(2\ell+1)} \prod_{j = 0}^{i-1} t_\ell\!\left(x^{2^j}\right) }{ \prod_{k = 0}^i \left(1 - qx^{2^k}\right)}.
\end{eqnarray*}
Now $\hat{f}_\ell(q^{-2})$ gives an upper bound for the limit (as word-length approaches infinity) of the probability that a word is a $Z_n$-instance of the form $LALBLAL$ with $|L| = \ell$. 

Taking this one step further, for some $Z_i$-instance $K$ of length $\ell_i$, the asymptotic probability that a word is a $Z_n$-instance constructed with $2^{n-i+1}$ copies of $K$ is at most
\begin{equation*} \label{eq:sum}
	\sum_{\ell_{i+1}=1}^\infty \cdots \sum_{\ell_{n-2}=1}^\infty\sum_{m=1}^\infty \hat{b}_{\ell_{i+1}}^{\ell_i} \cdots \hat{b}_{\ell_{n-2}}^{\ell_{n-3}}\hat{b}_{m}^{\ell_{n-2}} q^{-2m}.
\end{equation*}
Consequently,
\begin{eqnarray*}
	\II(Z_n,q) & \leq &\sum_{\ell_1=1}^\infty \cdots \sum_{\ell_{n-2}=1}^\infty\sum_{m=1}^\infty a_{\ell_1}\hat{b}_{\ell_2}^{\ell_1} \cdots \hat{b}_{\ell_{n-2}}^{\ell_{n-3}}\hat{b}_{m}^{\ell_{n-2}} q^{-2m}  \\
	& = &\sum_{\ell_1=1}^\infty \cdots \sum_{\ell_{n-2}=1}^\infty a_{\ell_1}\hat{b}_{\ell_2}^{\ell_1} \cdots \hat{b}_{\ell_{n-2}}^{\ell_{n-3}}\hat{f}_{\ell-2}(q^{-2}) . \\
\end{eqnarray*}

We need to get control of the tails to turn this into a computable sum.
A trivial upper bound for the asymptotic probability that a word is a $Z_n$-instance constructed with $2^{n-i}$ copies of $K$, and thus starts and ends with $K$, is $q^{-2\ell_i}$.
Since there are at most $q^{\ell_i}$ $Z_i$-instances of length $\ell_i$, the asymptotic probability that a word is a $Z_n$-instance with a $Z_i$-component of length $\ell_i$ is at most $q^{-\ell_i}$.
Therefore, the asymptotic probability that a word is a $Z_n$-instance with a $Z_i$-component of length greater than $N_i$ is at most
\begin{equation*} \label{eq:tail}
	\sum_{\ell_i = N_i+1}^\infty q^{-\ell_i} = \frac{q^{-N_1}}{q-1} .
\end{equation*}
Now in the upper bound of $\II(Z_n,q)$, we can replace the partial tail 
\[
	\sum_{\ell_1=1}^{N_1} \cdots \sum_{\ell_{i-1}=1}^{N_n}  \sum_{\ell_{i}=N_i+1}^\infty \sum_{\ell_{i+1}=1}^\infty \cdots \sum_{\ell_{n-2}=1}^\infty a_{\ell_1}\hat{b}_{\ell_2}^{\ell_1} \cdots \hat{b}_{\ell_{n-2}}^{\ell_{n-3}}\hat{f}_{\ell-2}(q^{-2})
\]
with 
\begin{eqnarray*}
	& & \sum_{\ell_1=1}^{N_1} \cdots \sum_{\ell_{i-1}=1}^{N_n} a_{\ell_1}\hat{b}_{\ell_2}^{\ell_1} \cdots \hat{b}_{\ell_{i-1}}^{\ell_{i-2}} \frac{q^{-N_1}}{q-1} \\
	& \leq & \left( \prod_{j = 1}^{i-1} N_j \right) \max_{\substack{\ell_j \leq N_j \\ 1 \leq j < i}}\left(a_{\ell_1}\hat{b}_{\ell_2}^{\ell_1} \cdots \hat{b}_{\ell_{i-1}}^{\ell_{i-2}}\right) \frac{q^{-N_1}}{q-1}\\
	& \leq & \left( \prod_{j = 1}^{i-1} N_j \right) q^{N_{i-1}} \frac{q^{-N_1}}{q-1}.
\end{eqnarray*}

Therefore,
\begin{eqnarray*}
	\II(Z_n,q) & \leq &\sum_{\ell_1=1}^{N_1} \cdots \sum_{\ell_{n-2}=1}^{N_n} a_{\ell_1}\hat{b}_{\ell_2}^{\ell_1} \cdots \hat{b}_{\ell_{n-2}}^{\ell_{n-3}}\hat{f}_{\ell-2}(q^{-2}) + \; \sum_{i = 1}^{n-2} \left( \left( \prod_{j = 1}^{i-1} N_j \right) q^{N_{i-1}} \frac{q^{-N_i}}{q-1} \right) .
\end{eqnarray*}

\appendix

\section{Proofs and Computations for Sections~\ref{IZ2} and~\ref{IZ3}}

\subsection{Proofs of Monotonicity}

\begin{lem} \label{lemF}
	For fixed $q \geq 2$, $\left\{|F(i)|\right\}_{i=0}^{\infty}$ is a decreasing sequence, where \[F(i) = F^q(i) = \frac{(-1)^jq^{1-2^i}}{\prod_{k=0}^{i} (1 - q^{1-2^k})}.\]
\end{lem}

\begin{proof}
	For $i>0$:
	\begin{eqnarray*}
		\frac{|F(i)|}{|F(i-1)|}  & = & \frac{q^{1 - 2^i}}{q^{1 - 2^{(i-1)}}\left(1 - q^{1 - 2^i}\right)}\\
		 & = & \frac{q^{-2^{(i-1)}}}{1 - q^{1 - 2^i}}\cdot \frac{1 + q^{1 - 2^i}}{1 + q^{1 - 2^i}}\\
		 & = & \frac{q^{-2^{(i-1)}}\left(1 + q^{1 - 2^i}\right)}{1 + q^{2 - 2^{i+1}}}\\
		 & < & \frac{(2)^{-2^{((1)-1)}}\left(1 + (2)^{1 - 2^{(1)}}\right)}{1 +(0)}\\
		 & = & 2^{-1}\left(1 + 2^{1 - 2}\right)\\
		 & < & 1.
	\end{eqnarray*}
\end{proof}

\begin{lem} \label{lemGH}
For fixed $\ell \geq 1$ and $q \geq 2$, $\left\{|G(i)|\right\}_{i=1}^{\infty}$ and $\left\{|H(i)|\right\}_{i = 1}^\infty$ are both decreasing sequences, where
\begin{eqnarray*}
	G(i)  = G_\ell^q(i) &=& \frac{ (-1)^i r\!\left(q^{-2^{i+1}}\right) \prod_{j = 0}^{i-1} s\!\left(q^{-2^{j+1}}\right)}{ \prod_{k = 0}^i \left(1 - q^{1-2^{k+1}}\right)};\\
	r(x) = r_\ell^q(x) & = & qx^{2\ell+1} - x^{4\ell} + x^{5\ell} - qx^{5\ell+1} +x^{6\ell};\\
	s(x) = s_\ell^q(x) & = & 1 - qx^{1-\ell} + x^{-\ell};\\
	H(i) = H_\ell^q(i) &=& \frac{ (-1)^i u\!\left(q^{-2^{i+1}}\right) \prod_{j = 0}^{i-1} v\!\left(q^{-2^{j+1}}\right) }{ \prod_{k = 0}^i \left(1 - q^{1-2^{k+1}}\right)};\\
	u(x) = u_\ell^q(x) & = & qx^{4\ell+1} - x^{5\ell} + qx^{5\ell+1} - x^{6\ell};\\
	v(x) = v_\ell^q(x) & = & 1 - qx^{1-\ell} + x^{-\ell} - qx^{1-2\ell} + x^{-2\ell}.
\end{eqnarray*}

\end{lem}

\begin{proof}
For $i>0$:
\begin{eqnarray*}
	\frac{|G(i)|}{|G(i-1)|}  & = & \frac{r\!\left(q^{-2^{i+1}}\right)}{r\!\left(q^{-2^i}\right)}\cdot \frac{s\!\left(q^{-2^i}\right)}{1 - q^{1-2^{i+1}}}\\
	& = & \frac{ q^{1-2^i(4\ell+2)} - q^{-2^i(8\ell)} + q^{-2^i(10\ell)} - q^{1-2^i(10\ell+2)} + q^{-2^i(12\ell)} }{ q^{1-2^i(2\ell+1)} - q^{-2^i(4\ell)} + q^{-2^i(5\ell)} - q^{1-2^i(5\ell+1)} + q^{-2^i(6\ell)} }\\
	& &\cdot \frac{ 1 - q^{1+2^i(\ell-1)} + q^{2^i\ell} }{ 1 - q^{1-2^i(2)} }\\
	& < & \frac{ q^{1-2^i(4\ell+2)} }{ q^{1-2^i(2\ell+1)} - q^{-2^i(4\ell)} }\cdot \frac{q^{2^i\ell} }{ 1 - q^{1-2^i(2)} }\\
	& = & \frac{ q^{1-2^i(3\ell+2)} }{ q^{1-2^i(2\ell+1)} - q^{-2^i(4\ell)} - q^{2-2^i(2\ell+3)} + q^{1-2^i(4\ell+2)}}\cdot \frac{ q^{-1+2^i(2\ell+1)} }{ q^{-1+2^i(2\ell+1)} }\\
	& = & \frac{ q^{-2^i(\ell+1)} }{ 1 - q^{-1-2^i(2\ell-1)} - q^{1-2^i(2)} + q^{2^i(2\ell+1)} }\\
	& < & \frac{ (2)^{-2^1((1)+1)} }{ 1 - (2)^{-1-2^1(2(1)-1)} - (2)^{1-2^1(2)} + 0 }\\
	& = & \frac{ 2^{-4} }{ 1 - 2^{-3} - 2^{-3} } < 1;\\
	\\
	\frac{|H(i)|}{|H(i-1)|} & = & \frac{u\!\left(q^{-2^{i+1}}\right)}{u\!\left(q^{-2^i}\right)}\cdot \frac{v\!\left(q^{-2^i}\right)}{1 - q^{1-2^{i+1}}}\\
	& = & \frac{ q^{1-2^i(8\ell+2)} - q^{-2^i(10\ell)} + q^{1-2^i(10\ell+2)} - q^{-2^i(12\ell)} }{ q^{1-2^i(4\ell+1)} - q^{-2^i(5\ell)} + q^{1-2^i(5\ell+1)} - q^{-2^i(6\ell)} }\\
	& &\cdot \frac{ 1 - q^{1+2^i(\ell-1)} + q^{2^i\ell} - q^{1+2^i(2\ell - 1)} + q^{2^i(2\ell)} }{ 1 - q^{1-2^i(2)} }\\
	& < & \frac{ q^{1-2^i(8\ell+2)} }{ q^{1-2^i(4\ell+1)} - q^{-2^i(5\ell)} }\cdot \frac{q^{2^i(2\ell)} }{ 1 - q^{1-2^i(2)} }\\
	& = & \frac{ q^{1-2^i(6\ell+2)} }{ q^{1-2^i(4\ell+1)} - q^{-2^i(5\ell)} - q^{2-2^i(4\ell+3)} + q^{1-2^i(5\ell+2)}}\cdot \frac{ q^{-1+2^i(4\ell+1)} }{ q^{-1+2^i(4\ell+1)} }\\
	& = & \frac{ q^{-2^i(2\ell+1)} }{ 1 - q^{-1-2^i(\ell-1)} - q^{1-2^i(2)} + q^{2^i(\ell+1)} }\\
	& < & \frac{ (2)^{-2^1(2(1)+1)} }{ 1 - (2)^{-1-2^1((1)-1)} - (2)^{1-2^1(2)} + 0 }\\
	& = & \frac{ 2^{-6} }{ 1 - 2^{-1} - 2^{-3} } < 1.
\end{eqnarray*}
\end{proof}

\subsection{Sage Code for Table~\ref{table:IZ3} of \texorpdfstring{$\II(Z_3,q)$-Values}{Values of the Asymptotic Density of Z3}} \label{P3code}

The following code to generate Table~\ref{table:IZ3} was run with Sage 6.1.1 \parencite{S-14}.


\begin{lstlisting}[frame=single]
# Calculate G(i), term i of expanded g(q^(-2)).
def r(L, q, x):
    X = x^L
    return q*x*X^2 - X^4 + X^5 - q*x*X^5 + X^6
def s(L, q, x):
    return 1 - q*x^(1 - L) + x^(-L)
def G(L, q, i):
    num = prod([s(L, q, q^(-2^(j + 1))) for j in range(i)])
    den = prod([1 - q^(1 - 2^(k + 1)) for k in range(i + 1)])
    return (-1)^i * r(L, q, q^(-2^(i + 1))) * num / den
# Calculate H(i), term i of expanded h(q^(-2)).
def u(L, q, x):
    return q*x^(4*L + 1) - x^(5*L) + q*x^(5*L + 1) - x^(6*L)
def v(L, q, x):
    return 1 - q*x^(1 - L) + x^(-L) - q*x^(1 - 2*L) + x^(-2*L)
def H(L, q, i):
    num = prod([v(L, q, q^(-2^(j + 1))) for j in range(i)])
    den = prod([1 - q^(1 - 2^(k + 1)) for k in range(i+1)])
    return (-1)^i * u(L, q, q^(-2^(i + 1))) * num / den
# Generate the first N terms of {a_n}.
def a(q,N):
    A = [0, q]
    for n in range(2, N + 1):
        A.append(q*A[-1] - ((n + 1)%2)*A[floor(n/2)])
    return A
# Calculate the partial sum of I(Z_3, q).
def I(q, N, M):
    A, partial = a(q, N), 0
    for L in range(1, N+1):
        terms = [G(L, q, n) + H(L, q, n) for n in range(M + 1)]
        partial += A[L]*sum(terms)
    return partial
# Output bounds on I(Z_3, q) for small values of q.
N = 31 # Level of precision.
for q in range(2, 7):
	print 'q = %d:' %q
	L, U = round(I(q, N, 4), N), round(I(q, N, 5) + 2^(-N), N)
	print 'Lower bound with N = %d and M = 4:' %N, L
	print 'Upper bound with N = %d and M = 5:' %N, U
\end{lstlisting}

	\section{Word Trees Illustrating Theorem \ref{P3}} \label{TREES}

From Section \ref{IZ3}: ``For fixed bifix-free word $L$ length $\ell$, define $b_m^\ell$ to count the number of $Z_2$ words with bifix $L$ that are $Z_2$-bifix-free $q$-ary words of length $m$.''

In each of the following images, a word is struck through if it is not counted by $b_m$ but its descendants are. It is hashed through if its descendants are also eliminated.

\begin{figure}[ht]
\centering
\begin{threeparttable}

\begin{tabular}{l}

\begin{tikzpicture}[grow'=right,level distance = 53pt,sibling distance=-5pt,every node/.style={scale=.75}]
	\Tree [.{$b$_3^1 = 2} [.{$b$_4^1 = 3} [.{$b$_5^1 = 6} [.{$b$_6^1=14} [.{$b$_7^1=25} [.{$b$_8^1 = 52} {$b$_9^1 = 100} ] ] ] ] ] ];
\end{tikzpicture}
\\
\begin{tikzpicture}[grow'=right,level distance=55pt, sibling distance=-4pt,every node/.style={scale=.75}]
	\Tree [.000
			[.\sout{0000} 
				[.\sout{00000} 
					[.\xout{000000} ]
					[.000100 
						[.0000100 
							[.00000100 
								[.000000100 ]
								[.000010100 ]
							]
							[.00001100 
								[.000001100 ]
								[.000011100 ]
							]
						]
						[.0001100 
							[.00010100 
								[.000100100 ]
								[.000110100 ]
							]
							[.00011100 
								[.000101100 ]
								[.000111100 ]
							]
						]
					]
				]
				[.00100 
					[.001000 
						[.0010000 
							[.00100000 
								[.001000000 ]
								[.001010000 ]
							]
							[.00101000 
								[.001001000 ]
								[.001011000 ]
							]
						]
						[.0011000 
							[.00110000 
								[.001100000 ]
								[.001110000 ]
							]
							[.00111000 
								[.001101000 ]
								[.001111000 ]
							]
						]
					]
					[.001100 
						[.0010100 
							[.\sout{00100100} 
								[.\sout{001000100} ]
								[.001010100 ]
							]
							[.00101100 
								[.001001100 ]
								[.001011100 ]
							]
						]
						[.0011100 
							[.00110100 
								[.001100100 ]
								[.001110100 ]
							]
							[.00111100 
								[.001101100 ]
								[.001111100 ]
							]
						]
					]
				]
			] 
			[.0010
				[.00010 
					[.000010 
						[.0000010 
							[.00000010 
								[.000000010 ]
								[.000010010 ]
							]
							[.00001010 
								[.000001010 ]
								[.000011010 ]
							]
						]
						[.0001010 
							[.00010010 
								[.\sout{000100010} ]
								[.000110010 ]
							]
							[.00011010 
								[.000101010 ]
								[.000111010 ]
							]
						]
					]
					[.000110 
						[.0000110 
							[.00000110 
								[.000000110 ]
								[.000010110 ]
							]
							[.00001110 
								[.000001110 ]
								[.000011110 ]
							]
						]
						[.0001110 
							[.00010110 
								[.000100110 ]
								[.000110110 ]
							]
							[.00011110 
								[.000101110 ]
								[.000111110 ]
							]
						]
					]
				]
				[.00110 
					[.001010 
						[.\sout{0010010} 
							[.\xout{00100010} ]
							[.00101010 
								[.001001010 ]
								[.001011010 ]
							]
						]
						[.0011010 
							[.00110010 
								[.001100010 ]
								[.001110010 ]
							]
							[.00111010 
								[.001101010 ]
								[.001111010 ]
							]
						]
					]
					[.001110 
						[.0010110 
							[.00100110 
								[.001000110 ]
								[.001010110 ]
							]
							[.00101110 
								[.001001110 ]
								[.001011110 ]
							]
						]
						[.0011110 
							[.00110110 
								[.\sout{001100110} ]
								[.001110110 ]
							]
							[.00111110 
								[.001101110 ]
								[.001111110 ]
							]
						]
					]
				]
			 ]
		];
	\draw (1.62,0.13) rectangle (13,5.59);
	\draw (13,2.86)node[left]{$d_n^1$};
\end{tikzpicture}

\end{tabular}

\caption[Example word tree for Theorem \ref{P3} with $q=2, \ell = 1$.]{The `000' half of an example word tree for Theorem \ref{P3} with $q=2$, $L=\text{`0'}$, $\ell = |L| = 1$. The tree from LLLL counted by $d_n$ is boxed.} \label{figure:T21}

\end{threeparttable}
\end{figure}
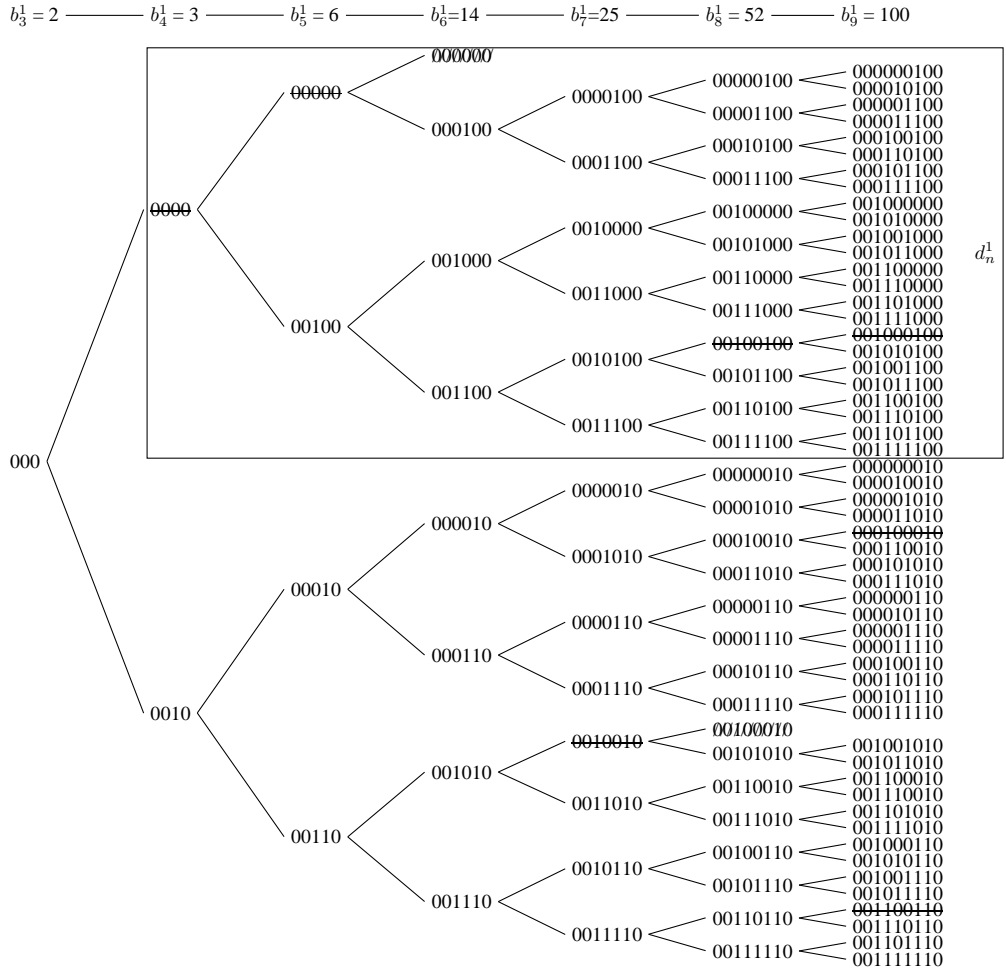

	\begin{figure}[ht]
\centering
\begin{threeparttable}

\begin{tabular}{l}

\begin{tikzpicture}[grow'=right,level distance = 60pt,sibling distance=-5pt,every node/.style={scale=.75}]
	\Tree [.{$b$_5^2 = 2} [.{$b$_5^2 = 4} [.{$b$_7^2 = 8} [.{$b$_8^2=13} [.{$b$_9^2=32} {$b$_{10}^2=58} ] ] ] ] ];
\end{tikzpicture}
\\
\begin{tikzpicture}[grow'=right,level distance=62pt, sibling distance=-5pt, every node/.style={scale=.75}]
	\Tree [.01001
			[.010001 
				[.0100001 
					[.01000001 
						[.010000001 
							[.0100000001 ]
							[.\sout{0100010001} ]
						]
						[.010010001 
							[.0100100001 ]
							[.0100110001 ]
						]
					]
					[.\sout{01001001} 
						[.010001001 
							[.0100001001 ]
							[.0100011001 ]
						]
						[.010011001 
							[.\xout{0100101001} ]
							[.0100111001 ]
						]
					]
				]
				[.0101001 
					[.01010001 
						[.010100001 
							[.0101000001 ]
							[.0101010001 ]
						]
						[.010110001 
							[.0101100001 ]
							[.0101110001 ]
						]
					]
					[.01011001 
						[.010101001 
							[.0101001001 ]
							[.0101011001 ]
						]
						[.010111001 
							[.0101101001 ]
							[.0101111001 ]
						]
					]
				]
			]
			[.010101 
				[.0100101 
					[.01000101 
						[.010000101 
							[.0100000101 ]
							[.0100010101 ]
						]
						[.010010101 
							[.0100100101 ]
							[.0100110101 ]
						]
					]
					[.01001101 
						[.010001101 
							[.0100001101 ]
							[.0100011101 ]
						]
						[.010011101 
							[.0100101101 ]
							[.0100111101 ]
						]
					]
				]
				[.0101101 
					[.\sout{01010101} 
						[.010100101 
							[.0101000101 ]
							[.\sout{0101010101} ]
						]
						[.010110101 
							[.0101100101 ]
							[.0101110101 ]
						]
					]
					[.01011101 
						[.010101101 
							[.0101001101 ]
							[.0101011101 ]
						]
						[.010111101 
							[.0101101101 ]
							[.0101111101 ]
						]
					]
				]
			]
		] ;

	\draw (5.94,-2.12) rectangle (12.3,-1.38);
	\draw (12.3,-1.75) node[left]{$d_n^2$};
\end{tikzpicture}
\\
\begin{tikzpicture}[grow'=right,level distance=62pt, sibling distance=-5pt, every node/.style={scale=.75}]
	\Tree [.01101
			[.011001 
				[.0110001 
					[.01100001 
						[.011000001 
							[.0110000001 ]
							[.0110010001 ]
						]
						[.011010001 
							[.0110100001 ]
							[.0110110001 ]
						]
					]
					[.01101001 
						[.011001001 
							[.0110001001 ]
							[.\sout{0110011001} ]
						]
						[.011011001 
							[.0110101001 ]
							[.0110111001 ]
						]
					]
				]
				[.0111001 
					[.01110001 
						[.011100001 
							[.0111000001 ]
							[.0111010001 ]
						]
						[.011110001 
							[.0111100001 ]
							[.0111110001 ]
						]
					]
					[.01111001 
						[.011101001 
							[.0111001001 ]
							[.0111011001 ]
						]
						[.011111001 
							[.0111101001 ]
							[.0111111001 ]
						]
					]
				]
			]
			[.011101 
				[.0110101 
					[.01100101 
						[.011000101 
							[.0110000101 ]
							[.0110010101 ]
						]
						[.011010101 
							[.0110100101 ]
							[.0110110101 ]
						]
					]
					[.\sout{01101101} 
						[.011001101 
							[.0110001101 ]
							[.0110011101 ]
						]
						[.011011101 
							[.\xout{0110101101} ]
							[.0110111101 ]
						]
					]
				]
				[.0111101 
					[.01110101 
						[.011100101 
							[.0111000101 ]
							[.0111010101 ]
						]
						[.011110101 
							[.0111100101 ]
							[.0111110101 ]
						]
					]
					[.01111101 
						[.011101101 
							[.0111001101 ]
							[.\sout{0111011101} ]
						]
						[.011111101 
							[.0111101101 ]
							[.0111111101 ]
						]
					]
				]
			]
		]
\end{tikzpicture}

\end{tabular}

\caption[Example word tree for Theorem \ref{P3} with $q=2, \ell = 2$.]{Example word tree for Theorem \ref{P3} with $q=2$, $L=\text{`01'}$, $\ell = |L| = 2$. The tree from LLLL counted by $d_n$ is boxed.} \label{figure:T22}

\end{threeparttable}
\end{figure}
	\begin{figure}[ht]
\centering
\begin{threeparttable}

\begin{tabular}{l}

\begin{tikzpicture}[grow'=right,level distance = 59pt,sibling distance=-5pt,every node/.style={scale=.75}]
	\Tree [.{$b$_7^3 = 2} [.{$b$_8^3 = 4} [.{$b$_9^3 = 8} [.{$b$_{10}^3 = 16} [.{$b$_{11}^3 = 30} {$b$_{12}^3 = 63} ]]]]];
\end{tikzpicture}
\\
\begin{tikzpicture}[grow'=right,level distance=60pt,sibling distance=-3pt,every node/.style={scale=.75}]
	\Tree [.1000100 
			[.10000100 
				[.100000100 
					[.1000000100 
						[.10000000100 
							[.100000000100 ]
							[.100000100100 ]
						]
						[.10000100100 
							[.100001000100 ]
							[.100001100100 ]
						]
					]
					[.1000010100 
						[.10000010100 
							[.100000010100 ]
							[.100000110100 ]
						]
						[.10000110100 
							[.100001010100 ]
							[.100001110100 ]
						]
					]
				]
				[.100010100 
					[.1000100100 
						[.\sout{10001000100} 
							[.100010000100 ]
							[.100010100100 ]
						]
						[.10001100100 
							[.100011000100 ]
							[.100011100100 ]
						]
					]
					[.1000110100 
						[.10001010100 
							[.100010010100 ]
							[.100010110100 ]
						]
						[.10001110100 
							[.100011010100 ]
							[.100011110100 ]
						]
					]
				]
			]
			[.10001100 
				[.100001100 
					[.1000001100 
						[.10000001100 
							[.100000001100 ]
							[.100000101100 ]
						]
						[.10000101100 
							[.100001001100 ]
							[.100001101100 ]
						]
					]
					[.1000011100 
						[.10000011100 
							[.100000011100 ]
							[.100000111100 ]
						]
						[.10000111100 
							[.100001011100 ]
							[.100001111100 ]
						]
					]
				]
				[.100011100 
					[.1000101100 
						[.10001001100 
							[.100010001100 ]
							[.100010101100 ]
						]
						[.10001101100 
							[.100011001100 ]
							[.100011101100 ]
						]
					]
					[.1000111100 
						[.10001011100 
							[.100010011100 ]
							[.100010111100 ]
						]
						[.10001111100 
							[.100011011100 ]
							[.100011111100 ]
						]
					]
				]
			]
		]
\end{tikzpicture}
\\
\begin{tikzpicture}[grow'=right,level distance=60pt,sibling distance=-3pt,every node/.style={scale=.75}]
	\Tree [.1001100 
			[.10010100 
				[.100100100 
					[.1001000100 
						[.10010000100 
							[.100100000100 ]
							[.\sout{100100100100} ]
						]
						[.10010100100 
							[.100101000100 ]
							[.100101100100 ]
						]
					]
					[.1001010100 
						[.10010010100 
							[.100100010100 ]
							[.100100110100 ]
						]
						[.10010110100 
							[.100101010100 ]
							[.100101110100 ]
						]
					]
				]
				[.100110100 
					[.1001100100 
						[.10001000100
							[.100110000100 ]
							[.100110100100 ]
						]
						[.10011100100 
							[.100111000100 ]
							[.100111100100 ]
						]
					]
					[.1001110100 
						[.10011010100 
							[.100110010100 ]
							[.100110110100 ]
						]
						[.10011110100 
							[.100111010100 ]
							[.100111110100 ]
						]
					]
				]
			]
			[.10011100 
				[.100101100 
					[.1001001100 
						[.10010001100 
							[.100100001100 ]
							[.100100101100 ]
						]
						[.10010101100 
							[.100101001100 ]
							[.100101101100 ]
						]
					]
					[.1001011100 
						[.10010011100 
							[.100100011100 ]
							[.100100111100 ]
						]
						[.10010111100 
							[.100101011100 ]
							[.100101111100 ]
						]
					]
				]
				[.100111100 
					[.1001101100 
						[.\sout{10011001100} 
							[.100110001100 ]
							[.100110101100 ]
						]
						[.10011101100 
							[.100111001100 ]
							[.100111101100 ]
						]
					]
					[.1001111100 
						[.10011011100 
							[.100110011100 ]
							[.100110111100 ]
						]
						[.10011111100 
							[.100111011100 ]
							[.100111111100 ]
						]
					]
				]
			]
		];

	\draw (9.7,3.37) rectangle (12.05,3.66);
	\draw (12.05,3.52) node[left] {$d_n^3$};
\end{tikzpicture}

\end{tabular}

\caption[Example word tree for Theorem \ref{P3} with $q=2, \ell = 3$.]{Example word tree for Theorem \ref{P3} with $q=2$, $L=\text{`100'}$, $\ell = |L| = 3$. The tree from LLLL counted by $d_n$ is boxed.} \label{figure:T23}

\end{threeparttable}
\end{figure}
	\begin{figure}[ht]
\centering
\begin{threeparttable}

\begin{tabular}{l}

\begin{tikzpicture}[grow'=right,level distance = 104pt,sibling distance=-5pt,every node/.style={scale=.75}]
	\Tree [.{$b$_3^{1} = 3} [.{$b$_4^{1} = 8} [.{$b$_5^{1} = 24} {$b$_6^1 = 78} ] ] ];
\end{tikzpicture}
\\
\begin{tikzpicture}[grow'=right,level distance = 105pt,sibling distance=-5pt,every node/.style={scale=.75}]
	\Tree [.000 
			[.\sout{0000} 
				[.\sout{00000} 
					[.\xout{000000} ]
					[.000100 ]
					[.000200 ]
				]
				[.00100 
					[.001000 ]
					[.001100 ]
					[.001200 ]
				]
				[.00200 
					[.002000 ]
					[.002100 ]
					[.002200 ]
				]
			]
			[.0010 
				[.00010 
					[.000010 ]
					[.000110 ]
					[.000210 ]
				]
				[.00110 
					[.001010 ]
					[.001110 ]
					[.001210 ]
				]
				[.00210 
					[.002010 ]
					[.002110 ]
					[.002210 ]
				]
			]
			[.0020 
				[.00020 
					[.000020 ]
					[.000120 ]
					[.000220 ]
				]
				[.00120 
					[.001020 ]
					[.001120 ]
					[.001220 ]
				]
				[.00220 
					[.002020 ]
					[.002120 ]
					[.002220 ]
				]
			]
		];

	\draw (3.37,.89) rectangle (12.3,2.59);
	\draw (12.3,1.74) node[left]{$d_n^1$};
\end{tikzpicture}
\\
\begin{tikzpicture}[grow'=right,level distance = 105pt,sibling distance=-5pt,every node/.style={scale=.75}]
	\Tree [.010 
			[.0100 
				[.01000 
					[.010000 ]
					[.010100 ]
					[.010200 ]
				]
				[.01100 
					[.011000 ]
					[.011100 ]
					[.011200 ]
				]
				[.01200 
					[.012000 ]
					[.012100 ]
					[.012200 ]
				]
			]
			[.0110 
				[.\sout{01010} 
					[.\xout{010010} ]
					[.010110 ]
					[.010210 ]
				]
				[.01110 
					[.011010 ]
					[.011110 ]
					[.011210 ]
				]
				[.01210 
					[.012010 ]
					[.012110 ]
					[.012210 ]
				]
			]
			[.0120 
				[.01020 
					[.010020 ]
					[.010120 ]
					[.010220 ]
				]
				[.01120 
					[.011020 ]
					[.011120 ]
					[.011220 ]
				]
				[.01220 
					[.012020 ]
					[.012120 ]
					[.012220 ]
				]
			]
		]
\end{tikzpicture}
\\
\begin{tikzpicture}[grow'=right,level distance = 105pt,sibling distance=-5pt,every node/.style={scale=.75}]
	\Tree [.020 
			[.0200 
				[.02000 
					[.020000 ]
					[.020100 ]
					[.020200 ]
				]
				[.02100 
					[.021000 ]
					[.021100 ]
					[.021200 ]
				]
				[.02200 
					[.022000 ]
					[.022100 ]
					[.022200 ]
				]
			]
			[.0210 
				[.02010 
					[.020010 ]
					[.020110 ]
					[.020210 ]
				]
				[.02110 
					[.021010 ]
					[.021110 ]
					[.021210 ]
				]
				[.02210 
					[.022010 ]
					[.022110 ]
					[.022210 ]
				]
			]
			[.0220 
				[.\sout{02020} 
					[.\xout{020020} ]
					[.020120 ]
					[.020220 ]
				]
				[.02120 
					[.021020 ]
					[.021120 ]
					[.021220 ]
				]
				[.02220 
					[.022020 ]
					[.022120 ]
					[.022220 ]
				]
			]
		]
\end{tikzpicture}

\end{tabular}

\caption[Example word tree for Theorem \ref{P3} with $q=3, \ell = 1$.]{Example word tree for Theorem \ref{P3} with $q=3$, $L=\text{`0'}$, $\ell = |L| = 1$. The tree from LLLL counted by $d_n$ is boxed.} \label{figure:T31}

\end{threeparttable}
\end{figure}

%

	
\printbibliography

\end{document}